\documentclass{article}
\usepackage{amsmath,amsthm}
\usepackage{amsfonts}
\usepackage{amssymb}
\usepackage{epsfig}
\usepackage{srcltx}
\usepackage[all]{xy}

\newtheorem{defi}{Definition}[section]

\newtheorem{theo}[defi]{Theorem}
\newtheorem{lemma}[defi]{Lemma}
\newtheorem{prop}[defi]{Proposition}
\newtheorem{coro}[defi]{Corollary}
\newtheorem{thmbis}{Theorem}

\theoremstyle{remark}
\newtheorem{rem}[defi]{Remark}

\newcommand{\calR}{\mathcal{R}}
\newcommand{\calW}{\mathcal{W}}

\newcommand{\calT}{\mathcal{T}}

\newcommand{\Area}{{\rm Area}}
\newcommand{\length}{{\rm length}}

\makeatletter
\edef\@tempa#1#2{\def#1{\mathaccent\string"\noexpand\accentclass@#2 }}
\@tempa\rond{017}
\makeatother

\renewcommand{\phi}{\varphi} 
\newcommand{\m} {^{-1}}

\newcommand {\ra} {\rightarrow}

\renewcommand{\subsetneq}{\varsubsetneq}

\newcommand{\dunion}{\sqcup}
\newcommand{\disjoint}{\sqcup}

\newcommand{\ie} {i.~e.\ }

\newcommand {\cala} {{\mathcal {A}}}   
   
\newcommand {\calc} {{\mathcal {C}}}

\newcommand {\calg} {{\mathcal {G}}}

\newcommand {\calr} {{\mathcal {R}}}   
\newcommand {\cals} {{\mathcal {S}}}   
\newcommand {\calt} {{\mathcal {T}}}

\newcommand {\bbF} {{\mathbb {F}}}

\newcommand {\bbN} {{\mathbb {N}}}

\newcommand{\grp}[1]{\langle #1 \rangle}

\title{Presenting parabolic subgroups}
\author{Fran\c{c}ois Dahmani and Vincent Guirardel}

\date{}
\begin{document}

\maketitle

\begin{abstract}
Consider a relatively hyperbolic group $G$. We prove that if $G$ is finitely presented, so are its parabolic subgroups.
Moreover, a presentation of the parabolic subgroups can be found algorithmically  from
a presentation of $G$, a solution of its word problem,
and generating sets of the parabolic subgroups.
We also give an algorithm that finds parabolic subgroups in a given recursively enumerable class of groups.
\end{abstract}

Consider a relatively hyperbolic group $G$ with parabolic subgroups $H_1,\dots,H_n$.
It is well known that if each $H_i$ is finitely generated (or finitely presented), then so is $G$.
Osin showed conversely that if $G$ is finitely generated, then so are $H_1,\dots,H_n$ \cite[Prop. 2.27]{Osin_relatively}.
Whether finite presentation of $G$ implies finite presentation of $H_1,\dots, H_n$ is an important question 
raised by Osin in \cite[Problem 5.1]{Osin_relatively}.

On the algorithmic side, given a finite presentation of a relatively hyperbolic group $G$ and a generating set of the parabolic subgroups, 
can one find a presentation of the parabolic subgroups?

We give a positive  answer to these two questions. 

\begin{thmbis}\label{th_fp}
Let $G$ be a finitely presented group. Assume that $G$ is hyperbolic relative to $H_1,\dots,H_n$.
Then each $H_i$ is finitely presented.  
\end{thmbis}

\begin{thmbis}\label{th_algo1}
There exists an algorithm that takes as input a finite presentation of a group $G$, 
a solution to its word problem,  and a collection of finite subsets $S_1,\dots,S_n\subset G$, 
and that terminates if and only if $G$
is hyperbolic relative to $\grp{S_1},\dots,\grp{S_n}$.

In this case, the algorithm outputs a linear isoperimetry constant $K$ for the corresponding relative presentation, 
a finite presentation for each of the parabolic subgroups $\grp{S_i}$,
and says whether $G$ is properly relative hyperbolic relative to $\grp{S_1},\dots,\grp{S_n}$ (\ie $\grp{S_i}\subsetneq G$ for all $i$).
\end{thmbis}

In this statement, the linear isoperimetry constant $K$ is 
for the relative presentation $X_\infty$ as defined in Section \ref{sec_truncated}.

If one is not given generating sets of the parabolic subgroups, one can search for them,
and require that they lie in some recursively enumerable class of groups.

\begin{thmbis}\label{th_algo2}
 There exists an algorithm as follows. It takes as input a finite presentation of a group $G$, 
a solution for its word problem, 
and a recursive class of finitely presented groups $\calc$ (given by a Turing machine enumerating presentations of these groups).

It terminates if and only if $G$
is properly hyperbolic relative to subgroups that are in the class $\calc$.

In this case, the algorithm outputs an isoperimetry constant $K$, 
a generating set and a finite presentation for each of the parabolic subgroups.
\end{thmbis}

The Turing machine enumerating $\calc$ is a machine that enumerates some finite presentations,
each of which represents a group in $\calc$, and such that every group in $\calc$ has at least one presentation
that is enumerated.

This paper can be seen as a continuation, extension, and precision, on the form and 
the substance  of \cite{Dah_finding}. 
It is based on the analysis of some Van Kampen diagrams in different \emph{truncated} relative presentations.
The main tool is Proposition \ref{prop;reduction_double} which says that if some relative presentation does not
satisfy a linear isoperimetric inequality, then this shows up on some diagram of small area and small complexity.

Section \ref{sec_context} recalls definitions about isometric inequalities, introduces truncated relative presentations,
and defines the complexity of a diagram.
Section \ref{sec_diagrams} contains the main technical results.
Section \ref{sec_conseq} is devoted to corollaries. Theorems \ref{th_fp}, \ref{th_algo1}, and \ref{th_algo2} 
follow from Corollaries \ref{cor_lin2fp}, \ref{cor_detect1} and \ref{cor_search}.

\section{Context}\label{sec_context}
\subsection{Linear isoperimetric inequalities}
Consider a finitely generated group $G$, with an arbitrary (non necessarily finite) generating set $S$. A \emph{presentation} of $G$ over $S$ 
is a set $\calR\subset \bbF_S$ that normally generates the kernel of the natural map from the free group $\bbF_S$ to $G$. 
The elements of $\calR$ are called \emph{defining relations}, and we usually write $G=\langle S|\calR \rangle$. 

We say that a presentation is \emph{triangular} if every defining relation has length $2$ or $3$ as word over the alphabet $S^\pm$. 
If one allows to increase the generating set, it is not restrictive to consider triangular presentations: from an arbitrary finite 
presentation, one can construct 
 effectively a triangular one. 

Consider $w\in \bbF_S$, viewed as  a reduced word over the alphabet $S^\pm$. 
If $w$ represents the trivial element of $G$ (we write $w\stackrel{G}{=} 1$), the \emph{area} of $w$ for the presentation $G= \langle S|\calR\rangle$, 
denoted by $\Area(w)$, is the minimal number $n$ such that  $w$ is the product in $\bbF_S$ of $n$ conjugates of elements of $\calR$.

Given a word  $w$ such that $w\stackrel{G}{=} 1$,
a \emph{Van Kampen diagram} for $w$ 
over the  presentation $G= \langle S|\calR\rangle$, is a simply connected planar 2-complex such that oriented edges are labeled by elements of $S^\pm$, 
such that reversing the orientation changes the label to its inverse,
and such that every $2$-cell has its boundary labeled by a cyclically reduced word conjugate to an element of $\calR\cup\calR\m$, 
and such that the boundary of the diagram itself is labeled by $w$. 
Sometimes, we just say \emph{cell} instead of $2$-cell.
It is well known that  $\Area(w)$ is the minimal number of $2$-cells of  Van Kampen diagrams for $w$.
See \cite[Section 5.1]{LyndonSchupp} for more details.

An \emph{isoperimetric function} of a presentation  $\grp{S|\calR}$  is a function $f:\mathbb{N} \to \mathbb{N}$ such that for all $w\in \bbF_S$,   
$\Area(w) \leq f(\length(w) )$.
Note that if $S$ is infinite,  
there are infinitely many words of a given length, and it may happen that no such function (with finite values) exists.

Our approach is based on the fact that a group is relatively hyperbolic if and only if
it has a presentation of a particular kind with a linear isoperimetric function \cite{Osin_relatively}, see Theorem \ref{thm_osin} below.
Another  important fact is that the failure of a specific linear isoperimetric inequality can be observed in a set of words of controlled area
(Gromov \cite{Gromov_hyperbolic}, Bowditch \cite{Bow_subquadratic}, Papasoglu \cite{Papasoglu_subquadratic}).

\begin{theo}[\cite{Papasoglu_subquadratic}]\label{theo;pap}
Let $G = \langle S | \calR\rangle $ be an arbitrary (non necessarily finite) triangular presentation of an arbitrary group. 

Assume that there is a word $w$ over the alphabet $S^\pm$ such that $w\stackrel{G}{=} 1$  and $\Area(w) > K \length(w)$. Then there exists a word $w'$    over the alphabet $S^\pm$  such that $w'\stackrel{G}{=}  1$, and such that
\begin{itemize}
 \item $\Area(w') \in [\frac{K}{2}, 240 K]$
 \item  $\Area(w') > \frac{1}{2\times 10^4} \length(w')^2$.
\end{itemize}
\end{theo}

\subsection{Truncated and exact relative presentations}
\label{sec_truncated}

Since finite generation of a relatively hyperbolic group implies finite generation 
of its maximal parabolic subgroups \cite[Prop. 2.27]{Osin_relatively}, 
we always assume that relatively hyperbolic groups and their maximal parabolic subgroups are finitely generated.

Let  $G$ be a finitely presented group, and  $H_1, \dots, H_n$ be finitely generated subgroups of $G$. 
For each $i$, let  $S_i$ be a finite symmetric generating set of $H_i$. 
Consider a finite triangular presentation $G = \langle S | \calR \rangle $
where $S$ is a finite symmetric generating set of $G$ containing each $S_i$, 
and $\calR$ is a finite set of triangular relations over $S$.

To introduce \emph{truncated} relative presentations, we need auxiliary groups 
$\Tilde H_1,\dots \Tilde H_n$, with generating sets $\Tilde S_1,\dots,\Tilde S_n$,
and with epimorphisms $p_i:\Tilde H_i\ra H_i$ that map $\Tilde S_i$ bijectively to $S_i$.
Informally, $\Tilde H_i$ is a group obtained from a presentation of $H_i$ over $S_i$
by removing some relations.
\emph{Exact} relative presentations will correspond to the case where each $p_i$ is an isomorphism.

Let $\calT(\Tilde H_i)\subset \Tilde H_i^*$ be the multiplication table of $\Tilde H_i$,
\ie the set tuples of at most $3$ elements of $\bbF_{S_i}$ whose
product is trivial in $\Tilde H_i$. 
Thus, we have $(a,b,c)\in \calT(\Tilde H_i)$ if and only if $abc=1$ in $\Tilde H_i$.

Let $\hat S=S\disjoint \Tilde H_1\disjoint \dots \disjoint \Tilde H_n$.
To each element of $\hat S$ corresponds naturally an element of $G$ 
via the inclusion  $S\subset G$ or via $p_i$.
These elements of $G$ form a generating set, in general infinite. 

 Given the initial presentation $G=\grp{S|\calR}$, $H_i$ and its generating set $S_i$,
the auxiliary groups $\Tilde H_i$ and the epimorphisms $p_i:\Tilde H_i\ra H_i$,
we associate the  \emph{truncated relative presentation} of $G$ as follows:
\begin{equation}\label{eq;presentation} 
 G = \left\langle\  \hat S\ \big|\ \calR',\ (\calT(\Tilde H_i))_{i=1\dots n} \ \right\rangle
\end{equation}
where $\calr'$ consists of $\calr$ together with all two-letter relators of the form
$\tilde s\m p_i(\tilde s)$ for $\tilde s\in\Tilde S_i$, ($p_i(\tilde s)$ being an element of $S$).
Obviously, this infinite presentation is indeed a triangular presentation of $G$.

We say that this presentation is truncated because only the multiplication table of $\Tilde H_i$ is included,
and not the one of $H_i$ (although all relations of $H_i$ are consequences of $\calr'$).
We say that a truncated relative presentation as above is \emph{exact} if
for all $i$, $p_i:\Tilde H_i\ra H_i$ is an isomorphism.

We will be particularly interested in the following  one-parameter family of truncated relative presentations $X_\rho$.
Given $G,H_i,S_i$ as above, and $\rho\in \bbN\cup\{\infty\}$,
we define $\calr_\rho(S_i)$ be the set of all words of length $\leq \rho$ on $S_i$ that are trivial in $H_i$,
$\Tilde H_i=\grp{S_i | \calr_\rho(S_i)}$, and $p_i:\Tilde H_i\ra H_i$ the obvious epimorphism.
We define $X_\rho$ the truncated relative presentation (\ref{eq;presentation}) corresponding to this data.
In particular, $X_\infty$ is an exact relative presentation, and if all $H_i$ are finitely presented,
then $X_\rho$ and $X_\infty$ coincide (as presentations) for $\rho$ large enough.

\begin{theo}[{\cite[Th. 1.7, Def. 2.29]{Osin_relatively}}]\label{thm_osin}
$G$ is hyperbolic relative to $H_1, \dots, H_n$ if and only if the exact presentation $X_\infty$
satisfies a linear isoperimetric inequality.
\end{theo}

The subgroups $H_1,\dots,H_n$ of $G$ are called the \emph{maximal parabolic subgroups}.
Since there is no risk of confusion, we will simply call them \emph{parabolic subgroups}.

\begin{rem}
 Osin includes all words of any length in the multiplication table.
One easily checks that this does not change the result.
\end{rem}

In section \ref{sec_conseq},
we are going to prove that if $X_\infty$ satisfies a linear isoperimetric inequality, so does $X_\rho$
for $\rho$ large enough. This will easily imply that parabolic subgroups are finitely presented.

\subsection{Complexities}

Since $X_\rho$ is an infinite presentation, it is convenient to have a measure of complexity for letters and words on $\hat S$.
Recall that $\hat S=S\disjoint \Tilde H_1\disjoint \dots \disjoint \Tilde H_n$.
For $a\in \Tilde H_i$,
we denote by $|\Tilde a|_{\Tilde S_i}$ the word length of $a$ relative to the generating set $\Tilde S_i$.
We define the \emph{complexity} $\|a\|$ of $a\in \Hat S$ as $1$ if $a\in S$, and as $|a|_{\Tilde S_i}$ if $a\in \Tilde H_i$.

Given a word $w=a_1\cdots a_n$ over $\hat S$, we define 
\begin{itemize}
\item $\length(w)=n$
\item $\|w\|_1=\sum_{i=1}^n \|a_i\|$
\item $\|w\|_\infty=\max_{i=1}^n \|a_i\|$
\end{itemize}
Note that if $w$ is a one-letter word, then $||w||_1=||w||_\infty=||w||$.

Similarly, if $D$ is a diagram (or a path) whose edges are labeled by elements of $\Hat S$,
we define $\|D\|_1$ and $\|D\|_\infty$ as the sum and the maximum of the complexities of the labels of its edges.
For a labeled path $p$, $\length(p)$ denotes its number of edges,
and $\Area(D)$ denotes the number of $2$-cells of a diagram $D$.

\section{Diagrams}\label{sec_diagrams}

The goal of this section is 
to prove that if $X_\rho$ does not satisfy a linear isoperimetric inequality,
this shows up on diagrams of small area and small complexity (Proposition \ref{prop;reduction_double}).

\subsection{Vocabulary}
\newcommand{\thick}{{\mathrm{thick}}}
\newcommand{\Dthick}{D_{\mathrm{thick}}}

\paragraph{Thickness.}
Let $D$ be a Van Kampen diagram over the presentation $X_\rho$ ($\rho$ being fixed in $\bbN\cup\{\infty\}$). 
We denote by $\Dthick\subset D$ the union of all $2$-cells, and of all vertices and edges that
are contained in the boundary of a $2$-cell.
We say that $D$ is \emph{thick} if $D=\Dthick$ \ie if every edge lies in the boundary of a $2$-cell.

\paragraph{Clusters.}
We define cells of type $\calr'$ (resp. of type $\Tilde H_i$) as those labeled by a word of $\calr'$
(resp. by a word in $\calt_{S_i}(\Tilde H_i)$).
Note that two cells of type $\Tilde H_i$ and $\Tilde H_j$ cannot share an edge if $i\neq j$.

Two cells of the same type $\Tilde H_i$ 
and sharing an edge are said \emph{cluster-adjacent}.
A \emph{cluster} is an equivalence class for the  transitive closure of this relation.
All $2$-cells of a cluster have the same type $\Tilde H_i$, which we define as the type of the cluster.
We identify a cluster with the closure $C$ of the $2$-cells it is made of.
Note that clusters are contained in $\Dthick$.

If $C$ is a cluster, we denote by $\partial C$ (its boundary) the  union of closed edges of $C$ that are in only one $2$-cell of $C$.

\begin{rem}\label{rem;R'}
Note that for any cluster $C$, any edge in $\partial C\setminus \partial D$
has complexity $1$. Indeed, the $2$-cell of $D\setminus C$ containing this edge is labeled by a relator 
$\Tilde s\m p_i(\Tilde s)$ for some $\Tilde s\in \Tilde S_i$.
\end{rem}

\subsection{Simply connected clusters, standard filling}

\begin{figure}[htbp]
 \centering
 \includegraphics{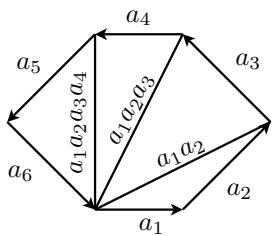}
 \caption{Standard filling.}
 \label{fig_standard}
\end{figure}

Note that a cluster $C$ (as a subset of the plane) is simply connected 
if and only if $C$ is a disk and $\partial C$ is an embedded circle in the plane.
We will mostly deal with diagrams whose clusters are simply connected.

Consider a simply connected cluster $C$, with $\partial C$ labeled by the cyclic word $a_1, \dots, a_n$ (where each $a_j\in \Tilde H_i$).
A \emph{standard filling} of $\partial C$ is a diagram with boundary $\partial C$,
with $n-2$ triangles as in figure \ref{fig_standard},
all whose vertices are in $\partial C$,
and whose interior edges are labeled by  $a_1\dots a_j$ for $j\leq n-2$, where $a_1\dots a_j$ 
is viewed as an element of $\Tilde H_i$.

\begin{lemma}\label{lem_std}
If $C$ is an arbitrary simply connected cluster, then $\|\partial C\|_1 \leq 3\Area(D) + \| \partial D\|_1$,  

If $C$ is standardly filled, then $\Area(C)=\length(\partial C)-2$,
and $\|C\|_\infty\leq \|\partial C\|_1$.
\end{lemma}

\begin{proof}
Let us partition $\partial C$ into edges that are in $\partial D$ and inner edges. 
There are at most  $\length(\partial C) \leq 3\Area (D)$ inner edges, each of which is of complexity $1$, by Remark \ref{rem;R'}. 
The sum of complexities of the edges in  $\partial D$ is bounded by  $\| \partial D\|_1$.  
This proves the first assertion.
The second assertion is clear  from the definition.
\end{proof}

\begin{rem}\label{rem;std} 
If $C$ is any cluster, then $\Area(C)\geq \length(\partial C)-2$.
Indeed, Denoting by $F$, $E_{int}$, $E_{ext}$ the number of $2$-cells, interior edges and boundary edges, 
connectedness of the dual graph implies $F-1\leq E_{int}$. 
Since cells of $C$ have at most $3$ sides, $ 2E_{int}+E_{ext}\leq 3F$.
It follows that  $E_{ext}\leq F+2$ as required.
\end{rem}

The following lemma shows that in many situations, clusters are simply connected.

\begin{lemma}\label{lem;sc_clusters}
Let $w$ be a word over $\hat S$ defining the trivial element in $G$.
Let $D$ be a minimal Van Kampen diagram for $w$ over the presentation $X_\rho$. 
Assume that $\rho\geq 3\Area (D)$.

If $D$ is chosen among diagrams for $w$ over $X_\rho$
to minimize successively the area, and the number of $2$-cells of type $\calR'$,
 then every cluster of $D$ is simply connected.

Assume either that $D$ is as above and that all its clusters are standardly filled,
or that $D$ minimizes successively the area,
 the number of $2$-cells of type $\calR'$ and $\|D\|_\infty$.
Then $$||D||_\infty\leq 3\Area(D)+||w||_1.$$
 \end{lemma}

\begin{proof} 
Assume by contradiction that there exists a cluster $C$ of type $\Tilde H_i$ that is not simply connected. 
Then there is a simply connected subdiagram $D'\subset D$ such that edges of $\partial D'$ are all in $\partial C\setminus \partial D$.
Since edges of $\partial D'$ lie in a $2$-cell, $\length(\partial D')\leq 3 \Area(D)$.
Moreover $\| \partial D' \|_\infty=1$,  since by Remark \ref{rem;R'}, every edge in $\partial C\setminus \partial D$ has complexity $1$.
Thus,  $\| \partial D' \|_1\leq 3 \Area(D)$.  
Since $\rho\geq 3\Area (D)$, the definition of $X_\rho$ says that the word labeled by $\partial D'$ is trivial in $\tilde H_i$. 
One can then replace the subdiagram bounded by $c$ by a diagram with same combinatorics, 
and with cells of type  $\Tilde H_i$. This contradicts the minimality of $D$ for the number of $2$-cells of type $\calR'$.
It follows that all clusters of $D$ are simply connected.

Assume now that all clusters are standardly filled.
By Lemma \ref{lem_std}, for each cluster $C$, $\|C\|_\infty\leq\|\partial C\|_1\leq 3\Area(D)+\|w\|_1$.
Since each edge of $\Dthick$ of complexity at least $2$ is contained in a cluster, this implies that $\|\Dthick\|_\infty\leq 3\Area(D)+\|w\|_1$.

Finally, assume that $D$ minimizes successively the area,
 the number of $2$-cells of type $\calR'$ and $|D||_\infty$.
Since  clusters of $D$ are simply connected, we can modify $D$ to
a diagram $D'$ whose clusters are standardly filled, and having the same area and the same number
 $2$-cells of type $\calr'$ as $D$. In particular, $||D||_\infty\leq ||D'||_\infty$. By the argument above,
$||D'||_\infty \leq  3\Area(D)+||w||_1 $ which concludes the proof.
\end{proof}

\subsection{Complicated clusters}

A cluster $C$ is said to be \emph{complicated} if $\partial C\cap \partial D$ contains at least two edges.

\begin{lemma}\label{lem;complexity_bounds}
 Assume that $D$ is a Van Kampen diagram, and $C\subset D$ is a simply connected cluster.

If $C$ is  not complicated, then
$\|\partial C\|_\infty \leq \length(\partial C)$, $\|\partial C\|_1 \leq 2 \length(\partial C)$. 
\end{lemma}

\begin{proof}
Denote by $\Tilde H_i$ the type of the cluster $C$,
so that edges of $C$ are labeled by elements of $\Tilde H_i$.
If $C$ is not complicated, all edges of $\partial C$ but one have complexity $1$.
The cluster being simply connected, the label of the remaining edge has the same image in $\Tilde H_i$ 
as a product of $\length(\partial C)-1$ elements of $S_i$. 
Therefore, this edge has complexity at most $\length(\partial C)-1$.
It follows that  $\|\partial C\|_\infty \leq \length(\partial C)$,
and  $\|\partial C\|_1 \leq  (\length(\partial C)-1) + \sum_{e\in \partial C} 1$. 
This proves the lemma. 
\end{proof}

\begin{lemma}[See also {\cite[Lemma 2.27]{Osin_relatively}}]\label{lem;bound_norm_infty}
Let $D$ be a Van Kampen diagram whose clusters are simply connected, non complicated, and standardly filled.

Then $\|\Dthick \|_\infty \leq 6 \Area (D) $.
\end{lemma}

\begin{proof}
Any edge of $\Dthick$ is either contained in a cell of type $\calR'$ (it has complexity $1$)
or in a cluster $C$. 
Since the number of edges of $D$ that lie in the boundary of a $2$-cell is bounded by $3\times \Area(D)$,
we have $\length(\partial C)\leq 3\times \Area(D)$.
Since $C$ is not complicated,
 $\|C\|_\infty \leq 6\times \Area(D)$ by Lemma \ref{lem;complexity_bounds}. The lemma follows. 
\end{proof}

\subsection{Arcs-of-clusters and pieces}

\begin{figure}[htbp]
 \centering
 \includegraphics[width=\textwidth]{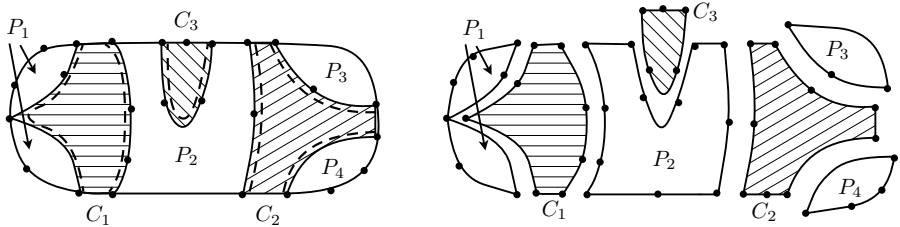}
 \caption{$3$ complicated clusters, $4$ regular pieces, and $6$ arcs-of-clusters}
 \label{fig_cluster}
\end{figure}

Consider a diagram $D$ whose clusters are simply connected. 
An \emph{arc-of-cluster}  is a maximal subpath  $c\subset\partial C$ for some complicated cluster $C$ that does not
contain any edge of $\partial D$ (see Figure \ref{fig_cluster}). 
Since $\partial C$ is an embedded circle, each arc-of-circle $c$ is an embedded arc with endpoints in $\partial D$,
and $c\cap\partial D$ contains no edge, but it may contain vertices distinct from its endpoints.

\newcommand{\Int}{\mathop{\mathrm{int}}}
We define \emph{regular pieces} of $D$ as the connected components of $D\setminus \rond\calc$ 
where $\rond\calc$
denotes the interior in $D$ of the union of all complicated clusters in $D$ 
(edges in $\partial D\cap \partial C$ for some complicated cluster are in  $\rond\calc$),
see Figure \ref{fig_cluster}.
Regular pieces and complicated clusters are called \emph{pieces}.

Here is an alternative definition. For each complicated cluster $C$, consider properly embedded arcs 
with endpoints in $\partial D$,  that are very close and parallel 
to each arc-of-cluster, obtained by pushing inside $C$ the arcs-of-clusters.
Let $\cala$ be the union of such embedded arcs when $C$ ranges over all complicated clusters.
Then connected components of $S\setminus \cala$ are in one-to-one correspondence with pieces.
On figure \ref{fig_cluster}, $\cala$ is represented by dotted lines.

Clearly,  the set of pieces induces a partition of the set of  $2$-cells of $D$. 
There is a natural  \emph{incidence graph} $\calg$ for this partition,
whose vertices are the pieces, whose edges are the arcs-of-clusters, the two endpoints of an edge being
the cluster and the regular piece on both sides of the corresponding arc-of-cluster.

\begin{lemma} \label{lem;arranging_pieces}
Let $D$ be a Van Kampen diagram, and 
assume that any cluster of $D$ is simply connected. 

The incidence graph $\calg$ is a bipartite tree 
and the degree of a vertex $v$ associated to a complicated cluster $C$ is at most the number of edges in $\partial D \cap \partial C$, 
with strict inequality when the vertex is $v$ is a leaf of the tree $\calg$.
\end{lemma}

\begin{proof}
The graph is bipartite by definition. It is connected because $D$ is.
Since every arc-of-cluster separates $D$,
every edge of the incidence graph disconnects it. This proves that $\calg$ is a tree.

Consider a vertex $v$ associated to a complicated cluster $C$. 
The  degree of $v$ is, by definition,  the number of arcs-of-clusters on $\partial C$. 
Since $C$ is simply connected, $\partial C$ is an embedded circle, and since $C$ is complicated, $\partial C$ contains
an edge of $\partial D$. 
By maximality in the definition of arc-of-clusters, each such arc is followed in $\partial C$ (with a chosen fixed orientation) by an edge of $\partial C \cap \partial D$. This association, which is clearly one-to-one, ensures the bound on the degree. 

 Finally, if $v$ is a leaf of $\calg$, its degree is $1$ and $\partial D\cap\partial C$ contains at least $2$ edges because $C$ is complicated.
\end{proof}

The following result of \cite{Dah_finding} was, to some extend, left to the reader. We include a proof.

\begin{lemma}\label{lem;bound_pieces_arcs}
Let $D$ be a Van Kampen diagram.   If every cluster is simply connected, 
then the number of pieces, and the number of arc-of-clusters are both bounded by $\length(\partial D)$. 
\end{lemma}

\begin{proof}
The number $N$ of pieces is the number of vertices of the incidence graph $\calg$.
Since $\calg$ is a tree, $N=E+1$ where $E$ is the number of edges of $\calg$, \ie the number of arcs-of-clusters.
Denote by $v_C$ the vertex corresponding to a cluster $C$, by $d(v_C)$ its degree, 
and by $V_{cl}$ the set of all vertices of $\calg$ corresponding to clusters.
Since $\calg$ is bipartite, 
$E= \sum_{v_C\in V_{cl}} d(v_C)$. 
By lemma \ref{lem;arranging_pieces}, $d(v_C)$ is bounded by the number $e(C)$ of edges of $\partial C\cap \partial D$. 
Therefore  $E\leq \sum_{v_C\in V_{cl}} e(C) \leq \length(\partial D)$. 

Finally, if some $v_C$ is a leaf of $\calg$, this last inequality is a strict inequality, which yields $N= E+1\leq \length(\partial D)$. 
There remains the case where some leaf of $\calg$ is a regular piece $B$. This means that $\partial B=\alpha\cup \beta$ where $\alpha$ is an
arc-of-cluster, and $\beta$ is a path in $\partial D$. Since clusters are simply connected, the endpoints of $\alpha$ are distinct, so $\beta$
contains at least an edge. This implies that $ \sum_{v_C\in V_{cl}} e(C) < \length(\partial D)$, and concludes the lemma.
\end{proof}

\subsection{Reduction to diagrams of small complexity}

We are now ready to state and prove the main statement of this section.
It claims that if $X_\rho$ does not satisfy a linear isoperimetric inequality,
this shows up on diagrams of small area (this is Papasoglu's theorem) and small complexity.

\begin{prop}[{\cite[Prop. 1.5]{Dah_finding}}]\label{prop;reduction_double}
 Let  $K\geq 10^6$ and   $\rho \in \mathbb{N} \cup \{\infty\}$,  $\rho \geq  3\times 240 K$.

 Assume that $X_\rho$ fails to satisfy a linear isoperimetric inequality  of constant $K$ 
(that is, there exists a word $w$ over  the alphabet $\hat S$ such that $\Area(w)>K\length(w)$).

 Then, there exists    a word $w''$ over the alphabet $\hat S$, and a minimal Van Kampen diagram 
 $D''$  (over $X_\rho$) for $w''$, such that
\renewcommand{\labelenumi}{(\arabic{enumi})}
 \begin{enumerate}
   \item $\Area(D'') \leq 240 K$
   \item $\| D'' \|_\infty \leq 2.10^6 K^2$
   \item  $\Area(D'') > \frac{\sqrt K}{600} \length(\partial D'').  $
 \end{enumerate}
\end{prop}

\begin{proof}
The first step is to apply Papasoglu's Theorem \ref{theo;pap} to the presentation $X_\rho$ 
to obtain a word $w'$ over $\Hat S$ 
for which  $K/2 \leq \Area(w') \leq 240 K$, and  $\Area(w') > \frac{1 }{2\times 10^4}  \length(w')^2$.  

Using $\sqrt{\Area(w')}>\frac{\length(\partial w')}{\sqrt{2\times 10^4}}$ and  $\Area(w')\geq K/2$,
we get
$$\Area(w') > \sqrt{ \frac{\Area(w')}{2\times 10^4}    }  \length(w')  
\geq  \frac{\sqrt{K}}{200}  \times \length(w'). $$

Choose a diagram $D'$ 
among minimal area diagrams over $X_\rho$ for $w'$ 
so that the number of $2$-cells of type $\calR'$ is minimal. 
We claim that up to changing $w'$, we can assume that $D'$ is thick
\ie all edges lie in the boundary of a $2$-cell.
Indeed, if all connected components
 $A'_1,\dots,A'_l$ of $D'_\thick$ satisfy
$\Area(A'_i) \leq  \frac{\sqrt{K}}{200}  \times  \length(\partial A'_i) $,
then $$\Area(D')=\sum_i \Area(A'_i) \leq \frac{\sqrt{K}}{200} \sum_i   \length(\partial A'_i)\leq \frac{\sqrt{K}}{200}\times \length(w')$$
a contradiction.
It follows that some component $A'_i$ satisfies 
$\Area(A'_i) >  \frac{\sqrt{K}}{200}  \times  \length(\partial A'_i) $.
Obviously, $\Area(A'_i)\leq \Area(D')\leq 240K$, and $A'_i$ is a diagram for $\partial A'_i$
that minimizes the area and the number of cells of type $\calr'$
(if not, substituting a diagram of smaller area for $\partial A'_i$ in $D'$ contradicts minimality of $D'$).
This proves that we can assume that $D'$ is thick. 

We do not have any control on the complexity of a diagram filling $w'$ yet.
By choice of $\rho$, Lemma \ref{lem;sc_clusters} shows that the clusters of $D'$ are simply connected. 
We can modify $D'$ and assume that all clusters are standardly filled.
By Remark \ref{rem;std}, $D'$ still minimizes area and the number of cells of type $\calr'$.
By Lemma \ref{lem;bound_pieces_arcs}, the number of pieces  in the decomposition into complicated clusters and regular pieces
is at most $\length(\partial D')$.

\begin{figure}[htbp]
  \centering
  \includegraphics{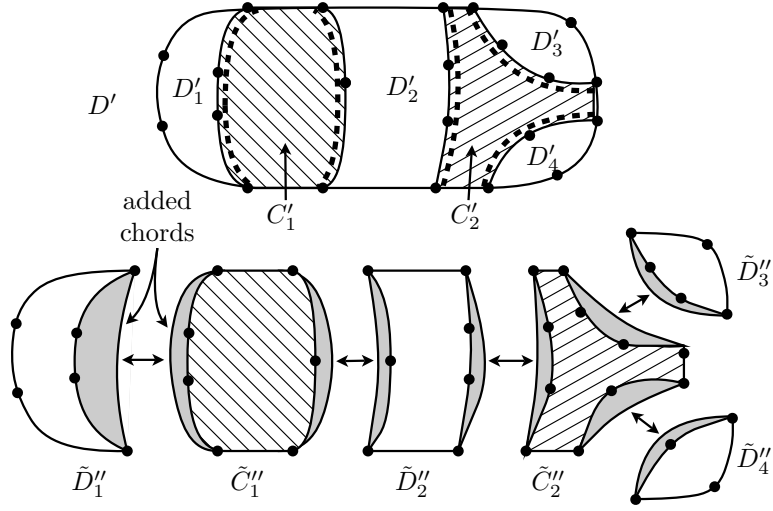}
  \caption{Adding chords to the pieces of $D'$, and regluing them together}
  \label{fig_chords}
\end{figure}

Let $C'_1, \dots, C'_s$ be the complicated clusters of $ D'$, and  $D'_1, \dots, D'_r$, be the regular pieces.  
We construct new diagrams $C''_i$, $D''_j$, and $\Tilde C''_i$, $\Tilde D''_j$ from $C'_i$, $D'_j$ by first \emph{adding chords},
then by changing the triangulation as follows (see Figure \ref{fig_chords}).

Fix a  complicated cluster $C'_k$ of $D'$, and denote by $\Tilde H_i$ its type.
Its boundary $\partial C'_k$ is
a union of pairwise disjoint arcs-of-clusters, together with arcs in $\partial  D'$.
Consider an arc-of-cluster $c\subset \partial C'_k$ whose edges are labeled by elements $a_1,\dots,a_n$ of $\Tilde H_i$,
and let $a_c=a_1\dots a_n\in \Tilde H_i$ be their product.
We glue along $c$ a standardly filled disk with boundary labeled by $a_1,\dots,a_n,a_c\m$.
We name the new edge labeled by $a_c\m$ a \emph{chord}. 
Performing this operation for each arc-of-cluster, we get a disk $C''_k$ made of cells of type $\Tilde H_i$.
Finally, we change the triangulation of this disk to a standard filling,
and we call $\Tilde C''_k$ the obtained diagram.
Note that $\Area (\Tilde C''_k) \leq \length(\partial \Tilde C''_k) -2$.

Now, we perform a similar operation for each regular piece $D'_j$.
For each arc-of-cluster $c\subset \partial D'_j$ labeled by $a_1,\dots,a_n\in \Tilde H_i$
(now, the type $\Tilde H_i$ may depend on $c$),
we define  $a_c=a_1\dots a_n\in \Tilde H_i$,
and glue to $C'_k$ along $c$ 
a new cluster of type $\Tilde H_i$, standardly filled, whose boundary
is labeled by $a_1,\dots,a_n,a_c^{-1}$.
Since the filling is standard, the area of the added cluster is $(n+1) -2 = \length(c)-1$. 
Performing this operation for each arc-of-cluster, we get the new diagram $D''_j$.
Finally, we take for $\Tilde D''_j$ a diagram with boundary $\partial D''_j$,
and minimizing successively the area and the number of $2$-cells of type $\calr'$.
\\

We are going to bound $\|\Tilde D''_j\|_\infty$ by first bounding $\|D''_j\|_\infty$.
Since all complicated clusters of $ D'$ are $C'_1,\dots, C'_s$,  $D''_j$ 
has no complicated cluster coming from $ D'$. The newly created clusters in $D''_j$
have just one edge in $\partial D''_j$, so are not complicated.
Therefore, clusters of $D''_j$ are not complicated, simply connected, and standardly filled.
Since $D'$ is thick, so is $D''_j$.
Applying Lemma \ref{lem;bound_norm_infty} to $D''_j$, we get 
$\|D''_j\|_\infty \leq 6\times \Area(D''_j) \leq 6\times 240K.$  

In particular, $\|\partial \Tilde D''_j\|_\infty=\|\partial D''_j\|_\infty \leq 6\times 240K$,
 and since $D''_j$ is thick, $\|\partial D''_j\|_1\leq 3\Area(D''_j)\|\partial D''_j\|_\infty
\leq 18\times (240K)^2$.
Applying Lemma \ref{lem;sc_clusters} to $\Tilde D''_j$, we get
$$\|\Tilde D''_j\|_\infty\leq 3\Area(D''_j)+\|\partial D''_j\|_1
\leq 3\times 240K+  18\times (240K)^2\leq 2.10^6 K^2.$$
This proves that for all  $j\in \{1,\dots, r\}$, $\Tilde D''_j$ satisfies assertions (1) and (2) of the proposition.

We now prove that one of the diagrams $\Tilde D''_j$, $j=1,\dots, r$ must satisfy (3).
Assume  by contradiction that for all $j\in\{1,\dots,r\}$, $\Area(\Tilde D''_j) \leq \frac{\sqrt{K}}{600} \length(\partial \Tilde D''_j)$. 
Note that  $\Tilde C''_k$ satisfies this inequality as well. 
Indeed, $\Area(\Tilde C''_k) \leq \length(\partial\Tilde C''_k) $, 
and by assumption, $K\geq 10^6$ so  $\frac{\sqrt{K}}{600}\geq 1$.

Gluing together the diagrams $\Tilde D''_1,\dots,\Tilde D''_r$ and $\Tilde C''_1,\dots,\Tilde C''_s$ 
pairwise along the two chords corresponding to a given arc-of-cluster as shown on Figure \ref{fig_chords},
we get another (non necessarily minimal) Van Kampen diagram $\Tilde D'$ for $ w'$.

We have 
\begin{eqnarray*}
\Area(D')&\leq& \Area(\Tilde D') = \sum_j \Area (\Tilde D''_j) \ + \   \sum_k \Area (\Tilde C''_k) \\
  & \leq &  \frac{\sqrt{K}}{600}\Big(\ \sum_j  \length(\partial \Tilde D''_j) \ + \ \sum_k  \length(\partial \Tilde C''_k) \Big)\\
  & \leq &  \frac{\sqrt{K}}{600}\Big( \length(\partial D') + 2 n_a \Big)
\end{eqnarray*}
where $n_a$ 
is the number of arcs-of-clusters in $D'$.  By lemma \ref{lem;bound_pieces_arcs},  $n_a \leq  \length(\partial D')$, 
so $\Area(D') \leq \frac{\sqrt{K}}{200} \times  \length(\partial D')$, 
thus contradicting the property of $D'$ established at the beginning of the proof. 
\end{proof}

\section{Consequences}\label{sec_conseq}

\begin{coro}\label{cor_lin_Xrho}
 Assume that $X_\infty$ satisfies a linear isoperimetric inequality of factor $K\geq 10^6$. 
Let $K'=(600K)^2$ and $\rho(K)=10^{26}K^5$.
Then for all $\rho \geq \rho(K)$,  $X_\rho$ satisfies a linear isoperimetric inequality of factor $K'$.
\end{coro}

Before proving the corollary, we need to relate more explicitly the presentations $X_\rho$ and $X_\infty$.
Consider $\Hat S_{\rho}=S\dunion \Tilde H_1\dunion\dots \dunion \Tilde H_n$ and
$\Hat S_{\infty}=S\dunion  H_1\dunion\dots \dunion H_n$ the corresponding generating sets.
The morphisms $p_i:\Tilde H_i\ra H_i$ induce an obvious map $p:\Hat S_\rho\ra \Hat S_{\infty}$ that is the identity on $S$
and maps $\Tilde H_i$ to $H_i$ through $p_i$. If $w=a_1\dots a_n$ is a word over $\Hat S_\rho$, we denote by $p(w)=p(a_1)\dots p(a_n)$
the corresponding word over $\Hat S_\infty$.
Clearly, if $w$ is any relator of $X_\rho$, $p(w)$ is a relator of $X_\infty$.
It follows that given any diagram $D$ over $X_\rho$ for a word $w$, one gets a new diagram $p_*(D)$ for $p(w)$ over $X_{\infty}$
by applying the map $p$ to all the labels of all edges  of $D$. 

On the other hand,  $p_i$ induces a bijection between the balls of radius $\rho/2$ of
$\Tilde H_i$ and $H_i$, whose inverse we denote by $p_i\m$. 
Similarly, we denote by $p\m$ the inverse of the restriction of $p:\Hat S_\rho\ra\Hat S_\infty$ to the set 
of elements of complexity at most $\rho/2$.
Now, if $a,b,c\in H_i$ are in the ball of radius $\rho/3$ of $H_i$ and satisfy $abc=1$ in $H_i$, 
then $p_i\m(a)p_i\m(b)p_i\m(c)=1$ in $\Tilde H_i$.
This means that if some diagram $D$ over $X_\infty$ for $w$ satisfies
$||D||_\infty\leq \rho/3$, then the diagram $p\m_*(D)$ (with obvious notations) 
is a diagram over $X_\rho$ for $p\m(w)$.

\begin{proof}[Proof of Corollary \ref{cor_lin_Xrho}]
Assume that  $X_\rho$ fails to satisfy the predicted isoperimetric inequality (of factor $K'$), and argue towards a contradiction. 
By Proposition \ref{prop;reduction_double}, there is a word $w''$ representing the trivial element, with a diagram $D''$, 
minimal over the presentation $X_\rho$, of area at most $240 K'$, 
and complexity $||D''||_\infty\leq 2.10^6 K'^2$ and such that $\Area(D'') > K \times \length (w'')$. 

Consider the map $p:\Hat S_\rho\ra \Hat S_\infty$ described above.
Choose $D_0''$ among  diagrams for $p(w'')$ in the presentation $X_\infty$,
in order to minimize successively the area, the number of $2$-cells of type $\calR'$, and the complexity $\|D''_0\|_\infty$.
Since $X_\infty$ satisfies a linear isoperimetric inequality of factor $K$,
$\Area(D_0'')<\Area(D'')\leq 240K'$.
By Lemma \ref{lem;sc_clusters},
$\|D_0''\|_\infty\leq 720K'+\|p(w'')\|_1$.
On the other hand, 
\begin{eqnarray*}
\|p(w'')\|_1&\leq &\|w''\|_1 \leq \length(w'')||D''||_\infty \leq \frac{1}{K}\Area(D'')\times 2.10^6 K'^2\\
&\leq& \frac{240K'\times 2.10^6 K'^2}{K}
\leq 3.10^{25}K^5.
\end{eqnarray*}
Since $K\geq 10^6$, $720K'\leq 10^9 K^2 \leq K^5$.
By hypothesis on $\rho$, we see that $\|D_0''\|_\infty\leq \rho/3$.
It follows that $p\m_*(D)$ is a diagram over $X_\rho$ for $w''$, of area $<\Area(D'')$, a contradiction.
\end{proof}

\begin{lemma}\label{lem_lin2fp}
 Assume that $X_\rho$ satisfies a linear isoperimetric inequality of factor $K'$
with $\rho\geq \max(3K',2)$.

Then $p_i:\Tilde H_i\ra H_i$ is an isomorphism. In particular, $H_i$ is finitely presented,
with a presentation whose defining relations are of length $\leq  \rho$.
\end{lemma}

\begin{proof}
Assume by contradiction that $p_i:\Tilde H_i\ra H_i$ is not injective, and consider $a\in\ker p_i\setminus\{1\}$.
Then $a$ is a generator of the presentation $X_\rho$ that represents the trivial element of $G$.
Note that since $\rho>1$,  $a\notin \Tilde S_i$.
Therefore, there exists a  Van Kampen diagram $D$ over $X_\rho$ whose boundary consists of a single edge $e$ labeled $a$,
and whose area is at most $K'$. 
We choose a diagram for $a$ over $X_\rho$ 
in order to minimize successively the area, the number of $2$-cells of type $\calR'$, and $||D||_\infty$. 
Since  $\rho\geq 3K'$, Lemma \ref{lem;sc_clusters} implies that clusters of $D$
are simply connected. Since $a\notin \Tilde S_i$, $e$ lies in a cluster $C$ of type $\Tilde H_i$.
But since $C$ is simply connected, and since a cluster of type $\Tilde H_i$ involves only relations of $\Tilde H_i$,
we get that $a$ is trivial in $\Tilde H_i$, a contradiction.  
\end{proof}

\begin{coro}\label{cor_lin2fp}
 Assume that $X_\infty$ satisfies a linear isoperimetric inequality of factor $K$. 

 Then the subgroups $P_i$ are finitely presented, with a presentation whose defining relations are of length 
$\leq  \rho(\max(K,10^6))$.
\end{coro}

\begin{proof}
Without loss of generality, we can assume $K\geq 10^6$.
By Corollary \ref{cor_lin_Xrho}, $X_{\rho(K)}$ satisfies a linear isoperimetric inequality of factor  $K'= (600K)^2 $. 
Lemma \ref{lem_lin2fp} concludes.
\end{proof}

\begin{lemma}[see also {\cite[Lemma 5.4]{Osin_relatively}}]\label{lem_proper}
 Assume that $X_\infty$ satisfies a linear isoperimetric inequality of factor $K$. 

If $s\in S$ represents an element $a$ of $H_i$, then $||a||\leq 12K$.
\end{lemma}

\begin{proof}
The word $w=sa$ is a word of length $2$ over $X_\infty$. If it represents the trivial element in  $G$, then there is a
Van Kampen diagram $D$ over $X_\rho$ whose boundary is a path of length $2$ labeled $sa$,
and whose area is at most $2K$. 
We choose $D$ among minimal area diagrams over $X_\infty$ for $w$
so that the number of $2$-cells of type $\calR'$ is minimal. 
Since  $\rho=\infty$, Lemma \ref{lem;sc_clusters} implies that clusters of $D$
are simply connected, and we can assume that they are standardly filled. 

Note that there is no complicated cluster as only the edge labeled $a$ of $\partial D$ can be in a cluster.
By Lemma \ref{lem;bound_norm_infty}, this implies that $||D_\thick||_\infty\leq 12K$,
so $||a||\leq 12K$.
\end{proof}

We obtain the following improvement of \cite{Dah_finding}:

\begin{coro}\label{cor_detect1}
 There exists an algorithm that takes as input a finite presentation of a group $G$, 
a solution of its word problem,  and a collection of finite subsets $S_1,\dots,S_n\subset G$, 
and that terminates if and only if $G$
is hyperbolic relative to $\grp{S_1},\dots,\grp{S_n}$.

In this case, produces an isoperimetry constant $K$ for the presentation $X_\infty$, 
a finite presentation for each of the parabolic subgroups,
and says whether $G$ is parabolic (\ie $G=\grp{S_i}$ for some $i$).
\end{coro}

\begin{proof}
For a fixed  $K\geq 10^6$, we consider all diagrams $D$ over $X_\infty$ such that 
$\|D\|_\infty \leq B=2.10^6 K^2$ and $\Area(D) \leq 240 K$. 
There are only finitely many. The word problem in $G$ allows to list all relators of $\grp{S_i}$
of length at most $3B$, and hence to list these diagrams.
Out of this list, we make the list $\calW(K)$ of words labeling the boundaries of these diagrams.  

We claim that given $w\in \calW(K)$, we can compute $\Area(w)$.
Indeed, let $D'$ be a diagram for $w$ chosen to minimize area, the number of cells of type $\calr'$,
and $||D'||_\infty$. By Lemma \ref{lem;sc_clusters}, $||D'||_\infty\leq 3\Area(D')+||w||_1\leq 720K+||w||_1$.
We can compute an upper bound $M\geq 720K+||w||_1$ for $||D'||_\infty$,
and we can list all diagrams $D'$ with $\Area(D')\leq 240K$ and $||D'||_\infty\leq M$
whose boundary is $w$. We can then compute $\Area(w)$ as the minimal area of these diagrams.

Now we can check whether $\Area(w)\leq \frac{\sqrt{K}}{600}\length(w)$ for all $w\in\calW(K)$.
If this is not the case, the algorithm increments $K$ and starts over.

If this is the case, then by Proposition \ref{prop;reduction_double}, $X_\infty$ satisfies isoperimetric inequality of factor $K$, and the algorithm stops.
It outputs $K$, and gives as set of relators for $\grp{S_i}$, the set of all words of length $\leq \rho(K)$
that are trivial in $G$; this can be done using the word problem in $G$, and 
this is indeed a presentation of $\grp{S_i}$ by Lemma \ref{cor_lin2fp}. 
To check whether $G=\grp{S_i}$, one needs to check whether each $s\in S$ represents an element $a\in \grp{S_i}$.
Lemma \ref{lem_proper}  bounds the complexity of $a$, and we can try all possibilities for $a$ using the word problem.

If  $X_\infty$ does satisfy a linear isoperimetric inequality of factor $K_0$, then the process will obviously stop  when  $K$ will reach a value
greater than $ (600K_0)^2$. 
\end{proof}

\begin{coro}\label{cor_search}
 There exists an algorithm as follows. It takes as input a finite presentation of a group $G$, 
a solution for its word problem, 
and a recursive class of finitely presented groups $\calc$ (given by a Turing machine enumerating them).
It terminates if and only if $G$
is properly hyperbolic relative to subgroups that are in the class $\calc$.

In this case, the algorithm produces an isoperimetry constant $K$, a generating set and a finite presentation for each of the parabolic subgroups.
\end{coro}

\begin{proof}
First, enumerate all possible presentations of groups in $\calc$ using the Turing machine given as input, and Tietze transformations.
In parallel, list all possible families of finite subsets $\cals=(S_1,\dots,S_n)$ of $G$.
For each of them, run in parallel the algorithm of Corollary \ref{cor_detect1} that stops if $G$ is 
hyperbolic relative to $\grp{S_1},\dots,\grp{S_n}$ and outputs a presentation of $\grp{S_i}$ in this case,
and says whether $G$ is parabolic.
Get rid of those $\cals$ such that $G$ is parabolic.

Then stop if at some point, one sees that in some of the produced presentations, $\grp{S_i}$ lie in $\calc$.
\end{proof}


\begin{thebibliography}{Dah08}

\bibitem[Bow95]{Bow_subquadratic}
B.~H. Bowditch.
\newblock A short proof that a subquadratic isoperimetric inequality implies a
  linear one.
\newblock {\em Michigan Math. J.}, 42(1):103--107, 1995.

\bibitem[Dah08]{Dah_finding}
Fran{\c{c}}ois Dahmani.
\newblock Finding relative hyperbolic structures.
\newblock {\em Bull. Lond. Math. Soc.}, 40(3):395--404, 2008.

\bibitem[Gro87]{Gromov_hyperbolic}
M.~Gromov.
\newblock Hyperbolic groups.
\newblock In {\em Essays in group theory}, volume~8 of {\em Math. Sci. Res.
  Inst. Publ.}, pages 75--263. Springer, New York, 1987.

\bibitem[LS01]{LyndonSchupp}
Roger~C. Lyndon and Paul~E. Schupp.
\newblock {\em Combinatorial group theory}.
\newblock Classics in Mathematics. Springer-Verlag, Berlin, 2001.
\newblock Reprint of the 1977 edition.

\bibitem[Osi06]{Osin_relatively}
Denis~V. Osin.
\newblock Relatively hyperbolic groups: intrinsic geometry, algebraic
 properties, and algorithmic problems.
\newblock {\em Mem. Amer. Math. Soc.}, 179(843):vi+100, 2006.

\bibitem[Pap95]{Papasoglu_subquadratic}
Panagiotis Papasoglu.
\newblock On the sub-quadratic isoperimetric inequality.
\newblock In {\em Geometric group theory ({C}olumbus, {OH}, 1992)}, volume~3 of
 {\em Ohio State Univ. Math. Res. Inst. Publ.}, pages 149--157. de Gruyter,
 Berlin, 1995.

\end{thebibliography}

\begin{flushleft}
\noindent \textsc{Fran\c{c}ois Dahmani}\\
Institut Fourier,\\
Universit\'e Joseph Fourier (Grenoble 1)\\
BP 74, \\
F-38402 St Martin d'Hères c\'edex, France\\
{\tt francois.dahmani@ujf-grenoble.fr}
\\[5mm]

\noindent\textsc{Vincent Guirardel} \\
Institut de Recherche en Mathematiques de Rennes (IRMAR)\\
Universit\'e de Rennes 1\\
263 avenue du G\'en\'eral Leclerc, CS 74205\\
F-35042 Rennes c\'edex, France\\
\texttt{vincent.guirardel@univ-rennes1.fr}
\end{flushleft}

\end{document}